\documentclass[11pt, a4paper]{amsart}

\usepackage{amsmath,amssymb,amsthm}
\usepackage{hyperref}
\usepackage[dvipsnames,usenames]{color}
\usepackage[latin1]{inputenc}
\usepackage{graphicx}
\usepackage{psfrag}
\usepackage{color}
\usepackage{soul}

\newtheorem{guia}{}
\newtheorem{rem}{}

\newtheorem{teorema}[guia]{Theorem}

\newtheorem{coro}[guia]{Corollary}
\newtheorem{lema}[guia]{Lemma}

\newtheorem{obs}[rem]{\it Remark}

\newcommand{\be}{\beta}

\newcommand{\De}{\Delta}
\newcommand{\de}{\delta}
\newcommand{\ds}{\displaystyle}
\newcommand{\e}{\varepsilon}

\newcommand{\la}{\lambda}

\newcommand{\N}{\mathbb N}

\newcommand{\Om}{\Omega}

\newcommand{\p}{\partial}

\newcommand{\R}{\mathbb R}

\begin{document}

\title[Monotonicity of solutions in half-spaces]{Monotonicity of solutions for 
some nonlocal elliptic problems in half-spaces}

\author[B. Barrios, L. Del Pezzo, J. Garc\'{\i}a-Meli\'{a}n and A. Quaas]
{B. Barrios, L. del Pezzo, J. Garc\'{\i}a-Meli\'{a}n\\ and A. Quaas}

\date{}

\address{B. Barrios \hfill\break\indent
Departamento de An\'{a}lisis Matem\'{a}tico, Universidad de La Laguna
\hfill \break \indent C/. Astrof\'{\i}sico Francisco S\'{a}nchez s/n, 38200 -- La Laguna, SPAIN}
\email{{\tt bbarrios@ull.es}}

\address{L. Del Pezzo \hfill\break\indent
CONICET  \hfill\break\indent
Departamento de Matem\'{a}tica, FCEyN UBA
\hfill\break\indent Ciudad Universitaria, Pab I (1428)
\hfill\break\indent Buenos Aires,
ARGENTINA. }
\email{{\tt ldpezzo@dm.uba.ar}}

\address{J. Garc\'{\i}a-Meli\'{a}n \hfill\break\indent
Departamento de An\'{a}lisis Matem\'{a}tico, Universidad de La Laguna
\hfill \break \indent C/. Astrof\'{\i}sico Francisco S\'{a}nchez s/n, 38200 -- La Laguna, SPAIN
\hfill\break\indent
{\rm and} \hfill\break
\indent Instituto Universitario de Estudios Avanzados (IUdEA) en F\'{\i}sica
At\'omica,\hfill\break\indent Molecular y Fot\'onica,
Universidad de La Laguna\hfill\break\indent C/. Astrof\'{\i}sico Francisco
S\'{a}nchez s/n, 38200 -- La Laguna, SPAIN.}
\email{{\tt jjgarmel@ull.es}}

\address{A. Quaas\hfill\break\indent
Departamento de Matem\'{a}tica, Universidad T\'ecnica Federico Santa Mar\'{\i}a
\hfill\break\indent  Casilla V-110, Avda. Espa\~na, 1680 --
Valpara\'{\i}so, CHILE.}
\email{{\tt alexander.quaas@usm.cl}}


\begin{abstract}
In this paper we consider classical solutions $u$ of the semilinear 
fractional problem $(-\De)^s u = f(u)$ in $\R^N_+$ with $u=0$ in $\R^N \setminus \R^N_+$, 
where $(-\De)^s$, $0<s<1$, stands for the fractional laplacian, $N\ge 2$, $\R^N_+=\{x=(x',x_N)\in \R^N:\ 
x_N>0\}$ is the half-space and $f\in C^1$ is a given function.  
With no additional restriction on the function $f$, we show that bounded, nonnegative, nontrivial 
classical solutions are indeed positive in $\R^N_+$ and verify 
$$
\frac{\p u}{\p x_N}>0 \quad \hbox{in } \R^N_+.
$$
This is in contrast with previously known results for the local case $s=1$, where 
nonnegative solutions which are not positive do exist and the monotonicity property above 
is not known to hold in general even for positive solutions when $f(0)<0$.
\end{abstract}

\maketitle

\section{Introduction}
\setcounter{section}{1}
\setcounter{equation}{0}

The objective of the present paper is to deal with the semilinear problem
\begin{equation}\label{problema}
\left\{
\begin{array}{ll}
(-\De)^s u = f(u) & \hbox{in }\R^N_+,\\[0.35pc]
\ \ u=0 & \hbox{in }\R^N \setminus \R^N_+,
\end{array}
\right.
\end{equation}
where $N\ge 2$, $\R^N_+=\{x=(x',x_N)\in \R^N:\ x_N>0\}$ is the half-space and $f$ is a $C^1$ function. 
The operator $(-\De)^s$, $0<s<1$, is the well-known \emph{fractional laplacian}, which
is defined on smooth functions as
\begin{equation}\label{operador}
(-\De)^s u(x) =  \int_{\R^N} \frac{u(x)-u(y)}{|x-y|^{N+2s}} dy,
\end{equation}
up to a normalization constant which will be omitted for brevity. The integral in \eqref{operador} 
has to be understood in the principal value sense, that is, as the limit as $\e\to 0$ of the 
same integral taken in the complement of the ball $B_\e(x)$ of center $x$ and radius $\e$. Alternatively, this 
operator can be defined (omitting again the normalization constant) as 
\begin{equation}\label{operador2}
(-\De)^s u(x) =\frac{1}{2} \int_{\R^N} \frac{2u(x) -u(x+y)-u(x-y)}{|y|^{N+2s}} dy,
\end{equation}
where now the integral is absolutely convergent for sufficiently smooth functions.

\smallskip
\smallskip

Problems with nonlocal diffusion related to \eqref{problema} have been intensively 
studied in the last years, after their appearance when modeling different situations. 
For instance, anomalous diffusion and quasi-geostrophic flows, turbulence and
water waves, molecular dynamics and relativistic quantum mechanics of stars
(see \cite{BoG,CaV,Co,TZ} and references); or mathematical
finance (cf. \cite{A,Be,CoT}), elasticity problems \cite{signorini},
thin obstacle problem \cite{Caf79}, phase transition \cite{AB98, CSM05, SV08b},  
crystal dislocation \cite{dfv, toland} and stratified materials \cite{savin_vald}.

\medskip

Our inspiration to study problem \eqref{problema} comes from the local case 
$s=1$, that is 
\begin{equation}\label{problema-local}
\left\{
\begin{array}{ll}
-\De u = f(u) & \hbox{in }\R^N_+,\\[0.35pc]
\ \ u=0 & \hbox{on } \p \R^N_+.
\end{array}
\right.
\end{equation}
In a seminal series of papers (cf. \cite{BCN1, BCN2, BCN3, BCN4}), Berestycki, Caffarelli and 
Nirenberg obtained interesting 
qualitative properties for positive solutions of \eqref{problema-local} and Lipschitz nonlinearities 
$f$. Among other results, they showed that if $f(0)\ge 0$, then 
any positive solution of \eqref{problema-local} verifies
\begin{equation}\label{eq-monotonia-local}
\frac{\p u}{\p x_N}>0 \quad \hbox{in }\R^N_+
\end{equation}
(see \cite{BCN2} or \cite{BCN3}). This property had been shown initially with additional 
assumptions on both the solutions and the nonlinearities by Dancer in \cite{D1} and \cite{D2}.

The case $f(0)< 0$ is, however, more subtle, and only partial results are known for the 
moment. See \cite{BCN3}, \cite{FS} for the case $N=2$, \cite{FSo} for $N=2,3$ and 
\cite{CEGM} for dimensions $N\ge 2$. The main reason for this difference is the existence of 
nonnegative (periodic) solutions which are not strictly positive. 

\medskip

With regard to a similar property as \eqref{eq-monotonia-local} for solutions of the nonlocal 
problem \eqref{problema}, only some partial results have been achieved so far, at the best of 
our knowledge. Let us mention \cite{FW} and \cite{QX} where monotone, positive 
nonlinearities where considered, and \cite{chinos} for the particular case $f(t)=t^p$, $p>1$. 
On the other hand, the very recent preprint \cite{DSV} analyzes the same question in more 
general domains, but with a very restricted class of nonlinearities.

Our intention is to show that actually property \eqref{eq-monotonia-local} continues to be 
true for bounded, nonnegative, nontrivial solutions of \eqref{problema} \emph{with no additional 
assumptions placed on the nonlinearity} $f$ aside its regularity. This is in striking contrast 
with problem \eqref{problema-local}, where, as we have remarked, the case $f(0)<0$ 
remains unsolved in its full generality for the moment. 

\medskip

Throughout this work, we will deal with bounded, classical solutions of \eqref{problema}. 
However, this will not cause a loss in generality, since it is well-known from the regularity 
theory developed in  \cite{S, CS, CS2} and bootstrapping that bounded, viscosity solutions 
of \eqref{problema} are automatically classical. Observe that classical 
solutions verify $u\in C^{2s+\be}(\R^N_+)$ for every $\be\in (0,1)$, and in particular they are 
seen to be in $C^1(\R^N_+)$.

\medskip

Our main result is the following:

\begin{teorema}\label{teorema-1}
Assume $f\in C^1(\R)$ and let $u$ be a bounded, nonnegative, nontrivial classical solution of \eqref{problema}. 
Then $u$ is positive and 
\begin{equation}\label{eq-derivadapos}
\frac{\p u}{\p x_N}>0 \quad \hbox{in } \R^N_+.
\end{equation}
\end{teorema}

\bigskip

As a consequence of Theorem \ref{teorema-1}, we can also obtain some Liouville theorems for 
problem \eqref{problema} with some special nonlinearities.

\begin{teorema}\label{teorema-2}
Assume $f\in C^1(\R)$ is such that $f'(t)>0$ for $t>0$, and one of the following holds:

\begin{itemize}

\item[(a)] $f(0)\ne 0$;

\smallskip

\item[(b)] $f(0)=0$ and $f'(0)>0$.

\end{itemize}
Then problem \eqref{problema} does not admit bounded, nonnegative, nontrivial  
solutions.
\end{teorema}

\bigskip

An interesting particular case in Theorem \ref{teorema-2} is obtained when we set 
$f(t)=t-1$. In this case the differences between the local version \eqref{problema-local} and 
its nonlocal counterpart \eqref{problema} become more evident, since in the former 
there exists a unique nonnegative solution given by $u(x)=1-\cos x_N$ (see \cite{CEGM}), 
while for the latter we have:

\begin{coro}\label{coro-lineal}
The problem 
$$
\left\{
\begin{array}{ll}
(-\De)^s u = u-1 & \hbox{in }\R^N_+,\\[0.35pc]
\ \ u=0 & \hbox{in }\R^N \setminus \R^N_+
\end{array}
\right.
$$
does not admit any bounded, nonnegative, nontrivial solution.
\end{coro}

\medskip

It is interesting to remark that Theorem \ref{teorema-1} is an important tool to prove other 
Liouville theorems for bounded solutions of \eqref{problema}. Indeed, passing to the 
limit as $x_N\to +\infty$, we find that such solutions converge to a \emph{stable} solution 
of $(-\De)^s u=f(u)$ in $\R^{N-1}$. Then one can use the nonexistence theorems already 
available in that situation (cf. for instance \cite{DuS}).

\medskip

We conclude the introduction with a couple of comments on our proofs. We use  
moving planes to show that any nonnegative, bounded, classical solution of \eqref{problema} 
is monotone in the $x_N$ direction. To deal with the moving planes method, we mainly follow 
the approach in \cite{FW}. However, instead of representing the solutions of \eqref{problema} 
with the aid of Green's function in $\R^N_+$ at the onset, we use it for an adequate 
truncation related to $u$ and its reflections. This allows us to avoid any monotonicity or 
sign restriction on $f$. 

It is to be noted that at one point in the argument, when it is assumed that the moving of 
the planes stop somewhere, we need to rule out the existence of solutions which are 
symmetric with respect to a hyperplane contained in $\R^N_+$. In the local case, this is 
only possible under the additional restriction $f(0)\ge 0$, since such solutions do exist 
if $f(0)<0$. However, we show in the present work that symmetric solutions can not 
exist with no additional restriction on $f$ (see 
Theorem \ref{lema-noexistencia} below). In our opinion, this is a result of independent interest.
Its proof is based on the regularity inherited by symmetry in the strip, which allows to evaluate 
the equation on the boundary of the half space. Then the nonlocality of the operator implies,  
loosely speaking, that the interactions between points which are far away in $\R^N_+$ is too strong
and the solution must vanish. This is a remarkable difference with respect to the case $s=1$, 
where this interaction is not present.  

\medskip

The rest of the paper is organized as follows: in Section 2 we give some preliminaries 
and introduce the notation to be used for the moving planes. Section 3 deals with some 
properties of the Green's function in a half-space taken from \cite{FW} and with a 
different representation in terms of this function. In Section 4, we obtain a nonexistence 
result for bounded, nonnegative, nontrivial solutions which are symmetric with respect to 
a hyperplane, and in Section 5 we perform the proof of our main results, Theorems
\ref{teorema-1} and \ref{teorema-2}.

\bigskip

\section{Some preliminaries}\label{s2}
\setcounter{section}{2}
\setcounter{equation}{0}

In this section, we gather some preliminary properties which will be useful in the 
forthcoming sections. We notice that, although we are mostly concerned with solutions in 
the classical sense, other more general concepts of solutions have to be 
considered at some places in the present work, mainly due to the fact that we work with truncations 
of the original functions. 

Thus, throughout this section we will consider inequalities in the viscosity sense (see \cite{CS} 
for a definition). We begin by considering a version of the maximum principle for 
the operator $(-\De)^s$ in unbounded domains, which will be needed below. 
We believe that this result is new.

\begin{lema}\label{lema-pmax}
Assume $D\subset \R^N_+$ is a domain and let $u\in C(\R^N)$ be a bounded function verifying 
$(-\De)^s u\ge 0$ in $D$ in the viscosity sense, with $u\ge 0$ in $\R^N \setminus D$. Then 
$u\ge 0$ in $D$.
\end{lema}

\begin{proof}
First of all observe that the function $\varphi(x)=  (x_N)_+^s$ is $s-$harmonic in $\R^N_+$, 
where $(x_N)_+$ is the function which coincides with $x_N$ in $(0,+\infty)$ and vanishes in 
$(-\infty,0]$. Indeed, when $x\in \R^N_+$:
\begin{align*}
(-\Delta)^{s}\varphi (x) & =\int_{\mathbb{R}}{\frac{(x_N)^{s}-(y_N)_+^{s}}{|x_N-y_N|^{1+2s}} 
\hspace{-.5mm} \left(\int_{\mathbb{R}^{N-1}} \hspace{-1mm} 
{\frac{|x_N-y_N|^{1+2s}}{\left((x_N-y_N)^2+|x'-y'|^2\right)^{\frac{N+2s}{2}}}\, dy'} \hspace{-1mm} 
\right) \hspace{-1mm} dy_N}\\
&=\int_{\mathbb{R}}{\frac{(x_N)^{s}-(y_N)_+^{s}}{|x_N-y_N|^{1+2s}} \hspace{-.5mm} \left(\int_{\mathbb{R}^{N-1}} 
\frac{dz}{(1+z^2)^{\frac{N+2s}{2}}}\right)dy_N}\\
&=C \int_{\mathbb{R}} \frac{(x_N)^{s}-(y_N)_+^{s}}{|x_N-y_N|^{1+2s}} dy_N=0
\end{align*}
(see for instance the introduction in \cite{CJS} or Proposition 3.1 in \cite{ROS}). 

Next take $\e>0$ and consider the function
$$
v_\e(x)= u(x) + \e (x_N)_+^s, \ x\in \R^N_+.
$$
Since $u$ is bounded, there exists $M>0$ such that $v_\e \ge 0$ if $x_N\ge M$. Define the set 
$D_M=D \cap \{x\in \R^N:\ 0<x_N<M\}$. Then, in the viscosity sense,
\begin{equation}\label{later}
\left\{
\begin{array}{lll}
(-\De)^s v_\e \geq 0 & \hbox{in } D_M,\\[0.35pc]
\ \ v_\e  \ge 0 & \hbox{in }\R^N \setminus D_M.
\end{array}
\right.
\end{equation}
We are in a position to apply Theorem 2.3 in \cite{QX} to conclude that $v_\e \ge 0$ in $\R^N$. Letting 
$\e\to 0$, we obtain that $u\ge 0$ in $\R^N$. It is worth remarking that, although Theorem 2.3 in 
\cite{QX} is stated for functions which vanish in $\R^N\setminus D_M$, a careful inspection 
shows that it is still valid when the involved functions are nonnegative there.
\end{proof}

\bigskip

Before giving our next result, let us introduce some notation related to the method 
of moving planes. For $\la>0$ we denote, as customary:
$$
\begin{array}{l}
\Sigma_\la :=\{ x\in \R^N_+:\ 0<x_N<\la\}\\[.5pc]
T_\la :=\{x\in \R^N:\ x_N=\la\}\\[.5pc]
x^\la:=(x',2\la-x_N) \ \hbox{(the reflection of }x \hbox{ with respect to } T_\la).
\end{array}
$$
If $f$ is a given nonlinearity and $u$ stands for a a bounded, classical nonnegative 
solution of our problem \eqref{problema} we also set
$$
\begin{array}{ll}
u_\la(x)= \left\{
\begin{array}{ll}
u(x), & x\in \Sigma_\la \cup (\R^N \setminus \R^N_+)\\
u(x^\la), & x\not\in \Sigma_\la \cup (\R^N \setminus \R^N_+)
\end{array}
\right. \\[0.75pc]
w_\la(x)= u_\la(x)-u(x), \quad x\in \R^N.
\end{array}
$$
Since our ultimate objective is to show that $w_\la$ is always nonnegative in $\Sigma_\la$, the 
following will also be relevant:
$$
\begin{array}{l}
D_\la=\{x\in \Sigma_\la: \ w_\la(x)<0\}\\[0.5pc]
W_\la=\{x\in D_\la: \ f(u(x)) > f(u^\la(x))\}\\[0.5pc]
v_\la = w_\la \chi_{D_\la},
\end{array}
$$
where $\chi$ will stand throughout the paper for the characteristic function of a set.
It is plain that the function $v_\la$ will only be meaningful when $w_\la$ is negative somewhere 
in $\Sigma_\la$. We next state one of its important properties.

\begin{lema}\label{lema-signo}
Assume $w_\la<0$ somewhere in $\Sigma_\la$, for some $\la>0$. Then, 
\begin{equation}\label{desigualdad}
(-\De)^s v_\la \ge (f(u^\la)-f(u)) \chi_{D_\la} \quad \hbox{in } \R^N_+,
\end{equation}
in the viscosity sense.
\end{lema}

\begin{proof}
Let us prove first that, when $x\in D_\la$, \eqref{desigualdad} holds in the classical 
sense (cf. the proof of Theorem 1.1 in \cite{FeWa}). Denote
$$
z_\la = w_\la-v_\la.
$$
It is clear that in $D_\la$  \eqref{desigualdad} is equivalent to $(-\De)^s z_\la\le 0$. To prove this last 
inequality, denote by $E_\la$ the reflection through the hyperplane $T_\la$ of $D_\la$. Using that 
$z_\la \equiv 0$ in $D_\la$ and $z_\la \ge 0$ in $\Sigma_\la\setminus D_\la$, we have for every 
$x\in D_\la$:
\begin{align*}
(-\De)^s z_\la (x) &=-\left(  \int_{\Sigma_\la \setminus D_\la} + \int_{E_\la}+ \int_{\Sigma_\la^c \setminus E_\la} \right) 
\frac{z_\la(y)}{|x-y|^{N+2s}} dy\nonumber\\
& \leq  -\left(  \int_{\Sigma_\la \setminus D_\la} +  \int_{\Sigma_\la^c \setminus E_\la} \right) 
\frac{z_\la(y)}{|x-y|^{N+2s}} dy \nonumber\\
& = -  \int_{\Sigma_\la \setminus D_\la} z_\la(y) \left( \frac{1}{|x-y|^{N+2s}} - \frac{1}{|x-y^\la|^{N+2s}} \right)dy
\le 0.
\end{align*}
Here we have used that $|x-y|\le |x-y^\la|$ when $x\in D_\la$, $y\in \Sigma_\la$, which can be
easily checked. Thus \eqref{desigualdad} is proved in $D_\la$.

On the other hand, when $x\in \R^N_+\setminus \overline{D_\la}$, it is immediate that
$$
(-\De)^s v_\la (x)=- \int_{D_\la} \frac{v_\la (y)}{|x-y|^{N+2s}} dy \ge 0,
$$
since $v_\la=0$ in $\R^N_+\setminus D_\la$ and $v_\la<0$ in $D_\la$. Therefore \eqref{desigualdad} 
also holds in the classical sense in $\R^N_+\setminus \overline{D_\la}$. 

However, the function $v_\la$ needs not be smooth on $\p D_\la$, so that it is not to be 
expected that its fractional laplacian is even well-defined there. But the inequality can be 
checked in the viscosity sense. To prove this, take $x_0\in \p D_\la$ and let $\varphi\in C^\infty(\R^N)$ 
be such that $\varphi <v_\la$ in a reduced neighborhood $\mathcal{U}\setminus \{x_0\}$ of 
$x_0$, with $\varphi(x_0)=v_\la(x_0)=0$. Then $(-\De)^s v_\la (x_0)\ge 0$ means 
$(-\De)^s \psi (x_0) \ge 0$, where 
$$
\psi(x)= \left\{
\begin{array}{ll}
\varphi(x) & x\in \mathcal{U}\\
v_\la(x) & x\in \R^N\setminus \mathcal{U}
\end{array}
\right.
$$
(cf. \cite{CS}). The inequality $(-\De)^s \psi(x_0)\ge 0$ is easily checked since,
taking into account that $v_\la \le 0$ in $\R^N$, so that $\varphi\le 0$ in $\mathcal{U}$, 
we deduce
$$
(-\De)^s \psi (x_0)=- \int_{\mathcal{U}} \frac{\varphi(y)}{|x-y|^{N+2s}} dy
-  \int_{\R^N\setminus \mathcal{U}} \frac{v_\la (y)}{|x-y|^{N+2s}} dy\ge 0,
$$
as was to be shown.
\end{proof}

\medskip

\section{A representation in the half-space}
\setcounter{section}{3}
\setcounter{equation}{0}

In this section, we will show that the function $v_\la$ verifies an inequality 
which involves the Green's function in the half-space. As we have already remarked 
in the Introduction, the representation is rather general and does not impose any 
additional properties on the nonlinear term $f$. Recall that we are always assuming 
$0<s<1$.

We introduce the Green's function for $\R^N_+$ (see \cite{FW}). If $x,y\in \R^N_+$, we let
\begin{equation}\label{eq-green}
G_\infty^+(x,y) = \frac{k_N^s}{2} |x-y|^{2s-N} \int_0 ^{\psi_\infty^+(x,y)} \frac{t^{s-1}}{(t+1)^\frac{N}{2}} dt,
\end{equation}
where
$$
\psi_\infty^+ (x,y) = \frac{4 x_N y_N}{|x-y|^2}.
$$
In \eqref{eq-green}, $k_N$ is a positive constant whose actual value is immaterial for us. 
It is shown in Theorem 3.1 of \cite{FW} that if $u\in L^\infty(\R^N)$ vanishes outside $\R^N_+$ and  
$(-\De)^s u\in L^\infty(\R^N_+)$ is nonnegative, then 
$$
u(x) = \int_{\R^N_+} G_\infty^+ (x,y) (-\De)^s u(y) dy, \quad x\in \R^N_+.
$$
To avoid the sign restriction just mentioned, we follow a different approach. The information 
we obtain is slightly weaker, but it suffices for our arguments in the proofs of Section 5. Here 
is the main result of this section:

\begin{lema}\label{lema-varianteFW}
Assume $f$ is locally bounded and let $u$ be a nonnegative, bounded solution of \eqref{problema}. 
Suppose $w_\la<0$ somewhere in $\Sigma_\la$, for some $\la>0$. Then, for 
every $x\in W_\la=\{x\in D_\la: \ f(u(x)) > f(u^\la(x))\}$, 
$$
v_\la(x) \ge \int_{W_\la} G_\infty^+ (x,y) (f(u^\la(y))-f(u(y))) dy,
$$
where the integral is absolutely convergent.
\end{lema}

\medskip

In order to prove Lemma \ref{lema-varianteFW}, we borrow some notation and results 
from \cite{FW}. For $R>0$, define $B_R^+:=B_R (Re_N)\subset \R^N_+$, where $e_N$ stands for the 
last vector in the canonical basis, and let 
$$
G_R^+(x,y)= \frac{k_N^s}{2} |x-y|^{2s-N} \int_0 ^{\psi_R^+(x,y)} \frac{t^{s-1}}{(t+1)^\frac{N}{2}} dt,
$$
with 
$$
\psi_R^+(x,y) = \frac{(R^2 - |x-R e_N|^2)(R^2 - |y-R e_N|^2)}{R^2 |x-y|^2},
$$
be the Green's function in the ball $B_R^+$. We also introduce 
$$
\Gamma_R^+ (x,y)= C_{N,s} \left( \frac{R^2 - |x- Re_N|^2}{|y- Re_N|^2 - R^2}\right)^s |x-y|^{-N},
$$
the Poisson kernel for the same ball (cf. \cite{BGR} for some properties of both functions). According to 
Corollary 2.9 in \cite{FW}, if $h_R$ is the unique solution of the Dirichlet problem
$$
\left\{
\begin{array}{ll}
(-\De)^s h = g_1 & \hbox{in } B_R^+,\\[0.35pc]
\ \ h=g_2 & \hbox{in }\R^N \setminus B_R^+,
\end{array}
\right.
$$
where $g_1\in L^\infty(B_R^+)$ and $g_2\in L^\infty(\R^N\setminus B_R^+)$, then we can write:
\begin{equation}\label{eq-repres}
h_R (x) = \int_{B_R^+} G_R^+(x,y) g_1(y) dy + \int_{\R^N \setminus B_R^+} \Gamma_R^+(x,y) g_2(y) dy.
\end{equation}
Regarding this representation, it is to be noted that, when $g_2 \in L^\infty(\R^N)$, as a consequence 
of equation (3.7) in \cite{FW}, then 
\begin{equation}\label{eq-limite-gamma}
\lim_{R\to +\infty} \int_{\R^N \setminus B_R^+} \Gamma_R^+(x,y) g_2(y) dy=0
\end{equation}
for every $x\in \R^N_+$. 
The following properties of Green's function will be used in our proof of 
Lemma \ref{lema-varianteFW} and in the proof of Theorem \ref{teorema-1} in Section 5. 

\begin{lema}\label{lema-FW}
Fix $R_0>0$. Then the functions $G_R^+(x,y)$ are nondecreasing with respect to $R$ in 
$B_{R_0}^+ \times B_{R_0}^+$ if $R>R_0$ and verify 
$$
G_R^+ \to G_\infty^+  \quad \hbox{in }B_{R_0}^+\times B_{R_0}^+ \hbox{ as }R\to +\infty.
$$ 
Moreover, for every $\la>0$, there exists $C=C(N,s,\la)$ such that 
\begin{equation}\label{est-green}
G_\infty^+(x,y) \le C \min\{ |x-y|^{2s-N},|x-y|^{-N}\} \quad \hbox{for } x,y \in \Sigma_\la.
\end{equation}
In addition, the function $G_\infty^+(x,y)$ enjoys the following properties:

\begin{itemize}

\item[(a)] If $\{x_n\}$ is a bounded sequence, then for every $\la>0$
$$
\lim_{R\to +\infty} \int_{\Sigma_\la \cap B_R^c} G_\infty^+ (x_n,y) dy =0,
$$
uniformly in $n$.

\item[(b)] If $\la>0$ and $\{x_n\}$ is a bounded sequence, then for every $R>0$ there exists 
a positive constant $C$ such that
$$
\int_{\Sigma_\la \cap B_R} G_\infty^+(x_n,y) dy\le C \quad \hbox{for every } n\in \N.
$$

\item[(c)] For every $\la>0$ and $x_0\in \p \R^N_+$, we have 
$$
\lim_{x\to x_0} \int_{\Sigma_\la} G_\infty^+ (x,y) dy =0.
$$
\end{itemize} 
\end{lema}

\begin{proof}[Sketch of proof]
The statements about monotonicity and convergence of $G_R^+$ are a consequence 
of Lemma 3.2 in \cite{FW}. The estimates \eqref{est-green} follow because of Lemma 
4.1 there.

Parts (a) and (b) are a direct consequence of \eqref{est-green}, while for the proof of (c), 
we only have to notice that, if $h\in C(\R^N)$ is the unique solution of 
$$
\left\{
\begin{array}{ll}
(-\De)^s h = \chi_{\Sigma_\la}  & \hbox{in }\R^N_+,\\[0.35pc]
\ \ h=0 & \hbox{in }\R^N \setminus \R^N_+,
\end{array}
\right.
$$
then by Theorem 3.1 in \cite{FW} we have 
$$
h(x) =\int_{\Sigma_\la}G_\infty ^+(x,y) dy,
$$
and the proof follows because of the continuity of $h$ up to the boundary of 
$\R^N_+$. 
\end{proof}

\bigskip

We can now proceed to the proof of Lemma \ref{lema-varianteFW}. 

\medskip

\begin{proof}[Proof of Lemma \ref{lema-varianteFW}]
We start by observing that, by Lemma \ref{lema-signo}
\begin{equation}\label{eq4}
(-\De)^s v_\la \ge (f(u^\la)-f(u)) \chi_{D_\la} \quad \hbox{in } \R^N_+.
\end{equation}
Consider the balls $B_R^+$ introduced before and denote by $h_R$ the unique solution 
of the problem
\begin{equation}\label{prob-bola}
\left\{
\begin{array}{ll}
(-\De)^s h = (f(u^\la)-f(u)) \chi_{D_\la} & \hbox{in } B_R^+,\\[0.35pc]
\ \ h=v_\la & \hbox{in }\R^N \setminus  B_R^+.
\end{array}
\right.
\end{equation}
It is clear by \eqref{eq4} and the maximum principle that $v_\la \ge h_R$ in $B_R^+$. Therefore, 
according to \eqref{eq-repres}, we may write
\begin{equation}\label{comp-green}
\begin{array}{rl}
v_\la(x)  \ge h_R(x) \hspace{-2mm} & \ds = \int_{B_R^+\cap D_\la} G_R^+(x,y) (f(u^\la(y))-f(u(y))) dy \\[1pc]
&+\ds \int_{\R^N_+ \setminus B_R^+} \Gamma_R^+(x,y) v_\la(y) dy.
\end{array}
\end{equation}
Our intention is to pass to the limit in \eqref{comp-green}. Notice that, since $v_\la$ is 
bounded, we have by \eqref{eq-limite-gamma} that the last integral converges to zero as 
$R\to +\infty$. 

\smallskip

On the other hand, we obtain from Lemma \ref{lema-FW} that $G_R^+$ is nondecreasing 
as a function of $R$ and, for fixed $x\in \Sigma_{\la}$
$$
G_R^+(x,y) \le G_\infty ^+ (x,y) \le C \min\{|x-y|^{-N+2s},|x-y|^{-N}\} \in L^1(\Sigma_{\la}),
$$
as a function of $y$. Therefore, letting $R\to +\infty$ in \eqref{comp-green} and using dominated 
convergence we arrive at 
$$
\begin{array}{rl}
v_\la(x) \hspace{-2mm} & \ds \ge \int_{D_\la} G_\infty^+(x,y) (f(u^\la(y))-f(u(y))) dy \\[1pc]
& \ge \ds \int_{W_\la} G_\infty^+(x,y) (f(u^\la(y))-f(u(y))) dy,
\end{array}
$$
as was to be proved.
\end{proof}

\medskip

\begin{obs}\label{rem-semiespacio}
{\rm Similar results as the ones given in this section follow easily in other half-spaces by means of a 
simple change of variables. For instance, in $H:=\{x\in \R^N:\ x_N<\la\}$, the Green's 
function is given by
$$
G(x;y)=G_\infty^+ (x',\la-x_N;y',\la-y_N), \quad x,y\in H,
$$
and similar properties as those given in Lemma \ref{lema-varianteFW} are obtained at once.
}
\end{obs}

\bigskip

\section{A nonexistence theorem}
\setcounter{section}{4}
\setcounter{equation}{0}

In this section we will state and prove a nonexistence result for nonnegative solutions of \eqref{problema} 
which are symmetric with respect to a hyperplane. This result is fairly important in the moving 
planes argument, and it is the reason why the case $f(0)<0$ can be included in our 
proofs, in contrast with the local case $s=1$. We believe it is interesting in 
its own right.

It is to be noted that, when the nonlinearity verifies $f(0)\ge 0$, the nonexistence 
of these symmetric solutions is a direct consequence of the strong maximum principle. 
The proof we give, however, covers also this case. Observe that next theorem holds with 
minimal hypotheses on $f$.

\begin{teorema}\label{lema-noexistencia}
Assume $f$ is continuous at zero and let $u\in C^{2s+\be}(\R^N_+)$ ($0<\be<1$) be a bounded, 
nonnegative, classical solution of 
\eqref{problema} which is symmetric with respect to $T_\la$ in $\Sigma_{2\la}$ for some $\la>0$, 
that is
$$
u(x',2\la-x_N)=u(x',x_N), \quad x\in \Sigma_{2\la}.
$$
Then $f(0)=0$ and $u\equiv 0$ in $\R^N$.
\end{teorema}

\begin{proof}
We begin by showing that $u$ verifies the equation at $x=0$, that is,
\begin{equation}\label{eq-claim-1}
\int_{\R^N_+} \frac{u(y)}{|y|^{N+2s}} dy = -f(0).
\end{equation}
To see this, we first observe that with no loss of generality we may assume that $\be$ is 
restricted to satisfy $2s+\be<2$. Thus the 
symmetry of $u$ implies that the same regularity holds up to the boundary of $\R^N_+$, since 
necessarily $u$ and $\nabla u$ vanish there. Therefore 
$u\in C^{2s+\be}(\R^N)$.

Take an arbitrary sequence $\{x_n\} \subset \Sigma_{2\la}$ with $x_n\to 0$. Evaluating the 
equation at $x_n$, but using expression \eqref{operador2} for $(-\De)^s$, we see that
\begin{equation}\label{eq1}
f(u(x_n)) =\frac{1}{2} \int_{\R^N} \frac{2u(x_n)-u(x_n+y)-u(x_n-y)}{|y|^{N+2s}}\; dy.
\end{equation}
Now we have to distinguish between the cases $0<s<\frac{1}{2}$ and $\frac{1}{2}\le s<1$. 
In the former case, assuming $\be$ is such that $2s+\be<1$, we deduce from the regularity of $u$ that 
for sufficiently large $n$:
\begin{equation}\label{holder1}
|u(x_n)-u(x_n+y)| \le C |y|^{2s+\be} \quad \hbox{whenever } |y|\le 1,
\end{equation}
for some positive constant $C$. In the latter, if $\be$ is such that $2s+\be<2$, the regularity implies, 
also for large enough $n$
\begin{equation}\label{holder2}
|u(x_n)-u(x_n+y)- u(x_n-y)|  \le C |y|^{2s+\be} \quad \hbox{if } |y|\le 1.
\end{equation}
On the other hand, for $|y| \ge 1$,
\begin{equation}\label{holder3}
\left| \frac{u(x_n)-u(x_n+y)-u(x_n-y)}{|y|^{N+2s}}  \right| \le 3 \|u\|_{L^\infty (\R^N_+)} |y|^{-N+2s}.
\end{equation}
Inequalities \eqref{holder1}, \eqref{holder2} and \eqref{holder3} show that the integrand in \eqref{eq1} 
is bounded in absolute value by a function which is in $L^1(\R^N)$. Therefore, we may pass to the 
limit in \eqref{eq1} with the aid of dominated convergence to arrive at 
$$
f(0) =- \frac{1}{2} \int_{\R^N} \frac{u(y)+u(-y)}{|y|^{N+2s}}\; dy,
$$
which is equivalent to \eqref{eq-claim-1}.

By evaluating the equation at $x=2\la e_N$, using the fact that $u(2\la e_N)=0$ 
by symmetry, we deduce from \eqref{eq-claim-1} that 
\begin{equation}\label{eq2}
\int_{\R^N_+} \frac{u(y)}{|y|^{N+2s}} dy =\int_{\R^N_+} \frac{u(y)}{|2\la e_N-y|^{N+2s}} dy.
\end{equation}
Next, we split the second integral in two parts and use the symmetry of $u$ in $\Sigma_{2\la}$ 
to have:
\begin{align*}
\int_{\R^N_+} \frac{u(y)}{|2\la e_N-y|^{N+2s}} dy & = \left( \int_{\Sigma_{2\la}} +
 \int_{\R^N_+\setminus \Sigma_{2\la}} \right) \frac{u(y)}{|2\la e_N-y|^{N+2s}} dy\\[0.5pc]
& =\int_{\Sigma_{2\la}} \frac{u(z)}{|z|^{N+2s}} dz + 
 \int_{\R^N_+\setminus \Sigma_{2\la}} \frac{u(y)}{|2\la e_N-y|^{N+2s}} dy.
\end{align*}
Hence, from \eqref{eq2} we see that 
$$
\int_{\R^N_+ \setminus \Sigma_{2\la}} \frac{u(y)}{|y|^{N+2s}} dy = \int_{\R^N_+ \setminus \Sigma_{2\la}} 
\frac{u(y)}{|2\la e_N-y|^{N+2s}} dy.
$$
Taking into account that $u\ge 0$ and $|2\la e_N-y| \le |y|$ for $y\in \R^N_+ 
\setminus \Sigma_{2\la}$, we deduce $u\equiv 0$ in $\R^N_+ \setminus \Sigma_{2\la}$. 

Using this information and evaluating the equation at points $x\in \R^N_+\setminus \Sigma_{2\la}$, 
we obtain
\begin{equation}\label{eq3}
\int_{\Sigma_{2\la}} \frac{u(y)}{|x-y|^{N+2s}} dy =-f(0).
\end{equation}
Now observe that the integral above is a smooth function of $x$ if, say, $x_N\ge 2\la+1$, since 
the integrand is uniformly bounded and the integral is uniformly convergent at infinity when $x$ belongs to a 
compact set. Therefore, we are allowed to differentiate \eqref{eq3} with respect to $x_N$ to get:
$$
\int_{\Sigma_{2\la}} \frac{u(y)(x_N-y_N)}{|x-y|^{N+2s+2}} dy =0.
$$
However, $x_N-y_N\ge x_N-2\la\ge 1$ for $y\in \Sigma_{2\la}$ and the chosen values of $x$, so that the 
integrand is nonnegative and this gives $u\equiv 0$ in $\Sigma_{2\la}$, therefore in 
$\R^N$. It is clear that this can only happen when $f(0)=0$, and the proof is concluded.
\end{proof}

\bigskip

\section{Proof of the main results}
\setcounter{section}{5}
\setcounter{equation}{0}

In this final section we will prove our main contributions, Theorems \ref{teorema-1} and 
\ref{teorema-2}. The proof of Corollary \ref{coro-lineal} will not be given, since it is an 
immediate consequence of Theorem \ref{teorema-2}.

\bigskip

\begin{proof}[Proof of Theorem \ref{teorema-1}]
We have already said that the proof is an application of the method of moving planes, as used in 
\cite{QX} and \cite{FW}, but with some significant changes. In particular, we remark that we work 
with some truncations of the original functions, so we are led to the use of inequalities in the viscosity 
sense and Lemma \ref{lema-varianteFW}. We also need at some point the nonexistence result given 
by Theorem \ref{lema-noexistencia}. 

We follow the notation introduced in  Section 2.

\medskip

\noindent {\bf Step 1}. $w_\la \ge 0$ in $\Sigma_\la$ if $\la>0$ is small enough. 

\smallskip

To prove this, assume for a contradiction that $D_\la$ is not empty if $\la$ is small.  Since $u$ 
is bounded and $f$ is $C^1$, there exists a constant $L$ such that $f(u^\la)-f(u)\ge -L |u^\la-u|$. 
Therefore, using Lemma \ref{lema-signo} we have
$$
(-\De)^s v_\la \ge f(u^\la)-f(u) \ge -L |u^\la-u|= L v_\la \quad \hbox{in }
D_\la.
$$
By Theorem 2.4 in \cite{QX} we obtain $v_\la\ge 0$ in $D_\la$ when $\la$ is small enough, 
which is a contradiction. Therefore, $D_\la=\emptyset$ for small $\la$ and this shows the claim.

\medskip

\noindent {\bf Step 2.} Setting
$$
\la^*=\sup\{ \la>0: \ w_\mu\ge 0 \hbox{ in } \Sigma_\mu \hbox{ for every } \mu \in (0,\la)\},
$$
we have $\la^*=+\infty$. 

\smallskip

Assume again for a contradiction that 
$\la^*<+\infty$. Then there exists a sequence $\{\la_j\}$ of values such that $\la_j>\la^*$ for every 
$j$ and $\la_j\to \la^*$ as $j\to +\infty$, with $w_{\la_j}$ negative somewhere in $\Sigma_{\la_j}$. 
Consider the sets
$$
D_j= \{x\in \Sigma_{\la_j}: \ w_{\la_j} (x)<0\}
$$
and
$$
W_j=\{x\in D_j:\ f(u(x))> f(u^{\la_j}(x))\}.
$$
By the choice of $\la_j$, the sets $D_j$ are nonempty for every $j$. We claim that the 
same is true for $W_j$. Indeed, if we had $W_j=\emptyset$, then $f(u^{\la_j})\ge f(u)$ 
in $D_j$. Hence
$$
(-\De)^s v_{\la_j} \ge 0 \quad \hbox{in } D_j.
$$
By Lemma \ref{lema-pmax} we obtain $v_{\la_j}\ge 0$ in $D_j$, which is not possible. 
Therefore $W_j\neq \emptyset$. Thus it is possible to choose points $x_j\in W_j$ such that  
\begin{equation}\label{eq-eleccion}
v_{\la_j}(x_j) \le -\frac{1}{2} \| v_{\la_j} \|_{L^\infty(W_j)},
\end{equation}
and we can define the functions
$$
\widetilde u_j(x)= u(x'+x_j',x_N), \quad x\in \R^N_+.
$$
It is easily seen that  $\widetilde u_j$ is a solution of \eqref{problema}, verifying $\| \widetilde u_j \|_{L^\infty(\R^N_+)}=
\| u \|_{L^\infty(\R^N_+)}$. It is then standard, with the use of regularity theory, Ascoli-Arzel\'a's 
theorem and a diagonal argument, that for some subsequence
$$
\widetilde u_j \to \bar u
$$
locally uniformly in $\R^N$, where $\bar u$ is a nonnegative solution of 
\eqref{problema}. We may also assume that $x_{j,N}\to x_0\in [0,\la^*]$. Now three cases are 
possible:

\begin{itemize}

\item[(a)] $\bar u\not \equiv 0$, $x_0\in (0,\la^*)$;

\smallskip
\item[(b)] $\bar u\not \equiv 0$, $x_0=0$ or $x_0=\la^*$;

\smallskip
\item[(c)] $\bar u \equiv 0$.

\end{itemize}

\medskip
\noindent Before dealing with each of this cases, let us introduce some notation related to the 
functions $\widetilde u_j$. Let:
$$
\begin{array}{l}
\widetilde w_{\la_j}(x)= \widetilde u_j (x^{\la_j})- \widetilde u_j(x), \ \ x\in \Sigma_{\la_j},\\[0.5pc]
\widetilde{D}_j=\{x\in \Sigma_{\la_j}: \ \widetilde w_j(x)<0\},\\[0.5pc]
\widetilde{v}_{\la_j}(x) = \widetilde w_{\la_j}(x) \chi_{\widetilde D_j}(x), \ \ x\in \Sigma_{\la_j}, \\[0.5pc]
\widetilde{W}_j=\{x\in \widetilde{D}_j:\ f(\widetilde u_j(x))> f(\widetilde u_j(x^\la_j))\}
\end{array}
$$
(observe that $\widetilde{D}_j$ and $\widetilde{W}_j$ are nothing more than translations 
of $D_j$ and $W_j$, respectively). Denote also $z_j=(0,x_{j,N})$, $z_0=(0,x_0)$. 
By our choice of $x_j$ in \eqref{eq-eleccion} above, since it follows that $\widetilde{w}_{\la_j} (z_j)<0$ and $f(\widetilde u_j(z_j))> f(\widetilde u_j(z_j^{\la_j}))$, we deduce that $z_j \in \widetilde{W}_j$. Moreover we also get
\begin{equation}\label{eq-eleccion-2}
\widetilde v_{\la_j}(z_j) \le -\frac{1}{2} \| \widetilde v_{\la_j}\|_{L^\infty(\widetilde{W}_j)}.
\end{equation}
Now consider in turn each one of the cases (a), (b) and (c).

\medskip

In case (a), we see from $\widetilde u_j(x^{\la^*})\ge \widetilde u_j(x)$ in $\Sigma_{\la^*}$ that 
$\bar u(x^{\la^*}) \ge \bar u(x)$ in $\Sigma_{\la^*}$. Moreover, since $\widetilde u(z_j^{\la_j}) < 
\widetilde u(z_j)$,  we also have $\bar u(z_0^{\la^*}) = \bar u(z_0)$. 
Let us see that this implies $\bar u(x^{\la^*})\equiv \bar u(x)$ in $\R^N$. Indeed, arguing as in the 
proof of Lemma \ref{lema-signo} and denoting $\bar w_{\la^*}=\bar u ^{\la^*}-\bar u$, we obtain 
\begin{align*}
0 & =f(\bar u^{\la^*}(z_0))-f(\bar u(z_0))= (-\De)^s \bar w_{\la^*}(z_0)\\
& =-  \int_{\Sigma_{\la^*} \cup \R^N_-} \bar w_{\la^*}(y) \left( \frac{1}{|z_0-y|^{N+2s}} - \frac{1}{|z_0-y^\la|^{N+2s}} \right)dy,
\end{align*}
which implies $\bar w_{\la^*}\equiv 0$ in $\R^N$, since $\bar w_{\la^*}\ge 0$ in $\Sigma_{\la^*} \cup \R^N_-$ and 
$|z_0-y|\le |z_0-y^\la|$ for every $y\in \Sigma_{\la^*} \cup \R^N_-$. 

This means that $\bar u$ is symmetric with respect to $T_{\la^*}$, so that Theorem \ref{lema-noexistencia} 
implies $\bar u\equiv 0$,  a contradiction.

\medskip

As for case (b), assume $x_0=0$. We deduce from Lemma \ref{lema-varianteFW} and the choice of the points 
$z_j$:
\begin{align*}
\frac{1}{2} \| \widetilde v_{\la_j}\|_{L^\infty(\widetilde{W}_j)} \le - \widetilde v_{\la_j}(z_j) & 
\le L  \| \widetilde v_{\la_j}\|_{L^\infty(\widetilde{W}_j)} \int_{\widetilde{W}_j} G_\infty^+(z_j,y) dy\\[1pc]
& \le L  \| \widetilde v_{\la_j}\|_{L^\infty(\widetilde{W}_j)} \int_{\Sigma_{\la^*+1}} G_\infty^+(z_j,y) dy,
\end{align*}
where $L$ denotes a bound for the derivative of $f$ in $[0,\|u\|_{L^\infty(\R^N_+)}]$.
By Lemma \ref{lema-FW}, part (c), the last integral converges to zero since $z_j\to 0\in \p \R^N_+$. 
We deduce that $\widetilde v_{\la_j}=0$ in $\widetilde{W}_j$ when $j$ is large enough, a contradiction. 

When $x_0=\la^*$, a similar contradiction is reached. The only difference is that one now 
works with the Green's function in the half-space $\{x\in \R^N:\ x_N<\la_j\}$ (see Remark 
\ref{rem-semiespacio}). 

\bigskip

Finally, we consider case (c). This case can only arise when $f(0)=0$ and the proof is
different depending on the sign of $f'(0)$. We begin by assuming that $f'(0)>0$. 

In what follows, we denote by $\la_1(\Om)$ the principal eigenvalue of $(-\De)^s$ in $\Om$ under 
Dirichlet boundary conditions (cf. Proposition 9 in \cite{SV2}). If we take a ball $B_R$ with arbitrary 
center and radius $R$ then it can be seen by means of a simple scaling that 
$$ 
\la_1(B_R)=\frac{\la_1(B_1)}{R^{2s}}\to 0 \quad \hbox{as } R\to +\infty. 
$$
Thus it is possible to select a ball $B\subset \subset \R^N_+$ with the property that 
\begin{equation}\label{eq-autovalor}
\la_1(B)<\frac{1}{2} f'(0).
\end{equation}
Since $\widetilde u_j\to 0$ uniformly in $B$, we deduce
\begin{equation}\label{eq-ultima}
(-\De)^s \widetilde u_j = \frac{f(\widetilde u_j)}{\widetilde u_j}\; \widetilde u_j\ge \frac{1}{2}f'(0) \; 
\widetilde u_j \quad \hbox{in }B
\end{equation}
if $j$ is large enough. By Theorem 1.1 in \cite{QSX},  \eqref{eq-ultima} implies the opposite inequality in 
\eqref{eq-autovalor}, which is a contradiction.

Hence to conclude the proof only the case $f'(0)\le 0$ needs to be dealt with. Using the 
mean value theorem, we may write, for $y\in \widetilde W_j$:
$$
f(\widetilde u_j^{\la_j}(y))- f(\widetilde u_j(y)) = f'(\xi_j(y)) \widetilde v_j(y),
$$
where $\xi_j(y)$ is an intermediate value between $\widetilde u_j(y)$ and $\widetilde u_j^{\la_j}(y)$. 
Observe that $\xi_j\to 0$ uniformly on compact sets of $\R^N_+$ while 
$f'(\xi_j)\ge 0$. By Lemma \ref{lema-varianteFW} and \eqref{eq-eleccion-2}, we see that 
\begin{equation}\label{eq5}
\frac{1}{2} \| \widetilde v_{\la_j}\|_{L^\infty(\widetilde{W}_j)} \le  \| \widetilde v_{\la_j}\|_{L^\infty(\widetilde{W}_j)} 
\int_{\Sigma_{\la^*+1}} G_\infty^+ (z_j,y) f'(\xi_j(y)) dy.
\end{equation}
Now choose $R>0$ and split the integral in \eqref{eq5} in $B_R$ and $B_R^c$. If 
$L$ stands again for a bound for the derivative of $f$ in $[0,\|u\|_{L^\infty(\R^N_+)}]$, we have
\begin{align*}
\frac{1}{2} \| \widetilde v_{\la_j}\|_{L^\infty(\widetilde{W}_j)} & \le \| \widetilde v_{\la_j}\|_{L^\infty(\widetilde{W}_j)} 
\left(\int_{\Sigma_{\la^*+1} \cap B_R} G_\infty^+ (z_j,y)  f'(\xi_j(y)) dy \right.\\[1pc]
& \left. + L \int_{\Sigma_{\la^*+1} \cap B_R^c} G_\infty^+ (z_j,y) dy\right).
\end{align*}
Observe that the integral in $\Sigma_{\la^*+1}\cap B_R^c$ can be made as small as 
desired by taking $R$ large enough, thanks to the fact that $\{z_j\}$ is a bounded sequence 
and Lemma \ref{lema-FW}, part (a). Therefore the last term in the 
above inequality can be made less than $\frac{1}{4}$, say, if $R$ is chosen large. 
After we have fixed such a value of $R$, we observe that $\xi_j \to 0$ uniformly in 
$B_R$, so that $f'(\xi_j)\le o(1)$, since we are assuming $f'(0)\le 0$. 
Therefore, using Lemma \ref{lema-FW}, part (b), we arrive at 
$$
\frac{1}{4} \| \widetilde v_{\la_j}\|_{L^\infty(\widetilde{W}_j)} \le o(1) \| \widetilde v_{\la_j}\|_{L^\infty(\widetilde{W}_j)} 
$$
which, as above, is a contradiction.

\medskip

\noindent {\bf Step 3}. Proof of \eqref{eq-derivadapos}. 

\smallskip

As a consequence of steps 1 and 2, we have shown that $u^\la\ge u$ in $\Sigma_\la$ for every 
$\la>0$, that is, $u$ is nondecreasing as a function of the variable $x_N$. Since $u\in C^1(\R^N_+)$, 
this implies 
\begin{equation}\label{eq-der}
\frac{\p u}{\p x_N} \ge 0 \quad \hbox{in } \R^N_+.
\end{equation}
To conclude the proof of our theorem, we only have to show that the inequality is strict in 
\eqref{eq-der}. This is a consequence of the strong maximum principle for the derivative 
with respect to $x_N$. However, this function does not directly verify an equation in $\R^N_+$, 
since $u$ is not expected to be $C^1$ on $\p \R^N_+$. We overcome this difficulty by 
localizing the problem and working with incremental quotients.

Assume there exists $x_0\in \R^N_+$ such that 
\begin{equation}\label{eq-derivadacero}
\frac{\p u}{\p x_N} (x_0)=0.
\end{equation}
Choose $\de>0$ such that $B_{2\de} (x_0)\subset \subset \R^N_+$, and let $\phi \in C_0^\infty 
(B_{2\de} (x_0))$ be a cut-off function with the usual properties: $0\le \phi \le 1$ and $\phi=1$ in 
$B_{\de} (x_0)$. Choose a small $\tau>0$ and define for $x\in \R^N$:
$$
z_\tau(x)= \frac{u(x+\tau e_N)-u(x)}{\tau} \phi(x).
$$
Since $u$ is nondecreasing we have $z_\tau \ge 0$, and we obtain 
$$
(-\De)^s z_\tau \ge \frac{f(u(x+\tau e_N))-f(u(x))}{\tau} \quad \hbox{in } B_\de(x_0),
$$
with $z_\tau \ge 0$ in $\R^N$. Letting $\tau \to 0$, it is clear that 
$$
z_t\to z:= \frac{\p u}{\p x_N} \phi \quad \hbox{uniformly in }\R^N,
$$
and using Lemma 4.5 in \cite{CS} we see that 
$$
(-\De)^s z \ge f'(u)z \quad \hbox{in } B_\de(x_0),
$$
in the viscosity sense. The strong maximum principle and \eqref{eq-derivadacero} imply $z=0$ in 
$B_\de (x_0)$, so that $\frac{\p u}{\p x_N}=0$ in $B_\de(x_0)$. A standard connectedness argument 
then implies that 
$$
\frac{\p u}{\p x_N}=0 \quad \hbox{in } \R^N_+,
$$
which is impossible. This concludes the proof of \eqref{eq-derivadapos}. Observe by the way 
that $u>0$ in $\R^N_+$ is a direct consequence of \eqref{eq-derivadapos}.
\end{proof}

\bigskip

\begin{proof}[Proof of Theorem \ref{teorema-2}]
Assume $u$ is a nonnegative, bounded, nontrivial solution of \eqref{problema}. By Theorem 
\ref{teorema-1}, we have $u>0$ in $\R^N_+$. We first claim that for every $\de>0$, there exists $c(\de)>0$ 
such that 
$$
u(x) \ge c(\de) \quad \hbox{if } x_N\ge \de.
$$
Suppose for a contradiction that this is not true. Then there exists $\de>0$ and a sequence 
$\{x_n'\}\subset \R^{N-1}$ such that $u(x_n',\de)\to 0$. Define
$$
u_n(x) = u(x'+x_n',x_N), \quad x\in \R^N_+.
$$
Since $\{u_n\}$ is uniformly bounded, we obtain after passing to a subsequence that $u_n\to \bar u$ 
locally uniformly in $\R^N$, 
where $\bar u$ is a nonnegative, bounded solution of \eqref{problema} which verifies $\bar u(0,\de)=0$. 
Again by Theorem \ref{teorema-1} we see that $\bar u \equiv 0$. 

This is impossible if $f$ verifies (a) in the statement. When $f$ verifies (b), a similar argument as in 
case (c) in step 2 of the proof of Theorem \ref{teorema-1} also leads to a contradiction. This 
contradiction proves the claim.

Now fix any $\de>0$ and let 
$$
\theta =\inf_{c(\de)\le t \le M} f'(t)>0, 
$$
where $M= \| u\|_{L^\infty(\R^N_+)}$. If we choose any function $\phi \in C_0^\infty (\R^N_+\setminus 
\Sigma_\de)$ such that $0\le \phi\le 1$ and $\phi=1$ in $\R^N\setminus \Sigma_{2\de}$, we see 
as in step 3 in the proof of Theorem \ref{teorema-1} that the function
$$
z= \frac{\p u}{\p x_N} \phi
$$
verifies 
$$
(-\De)^s z \ge \theta z \quad \hbox{in }\R^N\setminus \Sigma_{2\de}.
$$
Arguing again with the principal eigenvalue $\la_1(B)$ in a sufficiently large ball $B$ contained in 
$\R^N\setminus \Sigma_{2\de}$ we reach a contradiction. This shows that no bounded, nonnegative, 
nontrivial solution to \eqref{problema} may exist under our hypotheses.
\end{proof}

\bigskip

\noindent {\bf Acknowledgements.} B. B. was partially supported by MEC Juan de la Cierva postdoctoral fellowship 
number  FJCI-2014-20504 (Spain). 
L. D. P. was partially supported by PICT2012 0153 from ANPCyT (Argentina).
B. B., J. G-M. and A. Q. were partially supported by Ministerio de Ciencia e
Innovaci\'on under grant MTM2014-52822-P (Spain). A. Q. was also partially supported by Fondecyt Grant
No. 1151180 Programa Basal, CMM. U. de Chile and Millennium Nucleus
Center for Analysis of PDE NC130017. B. B.  and J. G-M. would like to thank the Mathematics 
Department of Universidad T\'ecnica Federico Santa Mar\'ia where part of this work has been done 
for its kind hospitality.

\end{document}